\documentclass[12pt,oneside]{amsart}
\usepackage{amsmath}
\usepackage{amsthm}
\usepackage{amsfonts}
\usepackage{a4wide}
\usepackage{amssymb}
\newtheorem{thm}{Theorem}
\newtheorem{lemma}[thm]{Lemma}
\newtheorem{prop}[thm]{Proposition}
\newtheorem{corollary}[thm]{Corollary}

\theoremstyle{definition}
\newtheorem{defin}[thm]{Definition}

\theoremstyle{remark}
\newtheorem{rem}[thm]{Remark}
\newcommand{\DD}{{\mathbb D}}
\newcommand{\OO}{{\mathcal O}}

\newcommand{\EE}{{\mathbb E}}

\newcommand{\GG}{{\mathbb G}}

\newcommand{\CC}{{\mathbb C}}

\DeclareMathOperator{\Aut}{Aut} \DeclareMathOperator{\id}{id}
 
\DeclareMathOperator{\ord}{ord}

\DeclareMathOperator{\tr}{tr}

\renewcommand{\phi}{\varphi}


\begin{document}



\title{The Lempert theorem and the tetrablock}

\author{Armen Edigarian}

\address{Institute of Mathematics, Faculty of Mathematics and Computer Science, Jagiellonian University,
\L ojasiewicza 6, 30-348 Krak\'ow, Poland}
\email{{Armen.Edigarian}@im.uj.edu.pl}
\thanks{The work is partially supported by the grant of the Polish Minister for Science and Higher Education
No. N N201 361436.}

\thanks{Part of the paper was prepared while the stay of the Authors at the Erwin Schr\"odinger International Institute for Mathematical Physics in Vienna during the programme: {\it The dbar-Neumann problem: analysis, geometry and potential theory}.}

\author{\L ukasz Kosi\'nski}

\email{{Lukasz.Kosinski}@im.uj.edu.pl}

\author{W\l odzimierz Zwonek}

\email{{Wlodzimierz.Zwonek}@im.uj.edu.pl}





\begin{abstract}
In the paper we show that the Lempert property (i.e. the equality between the Lempert function and the Carath\'eodory distance) 
holds in the tetrablock, a bounded hyperconvex domain which is not
biholomorphic to a convex domain. The question whether such an equality holds was posed in  \cite{AWY}.
\end{abstract}

\maketitle

\section{Introduction and main results}
The paper may be seen as a direct continuation of the study of the geometry of the tetrablock, a domain introduced recently and then studied in a series of papers (\cite{AWY}, \cite{Abo 2007}, \cite{You 2007}, \cite{EZ}, \cite{Kos 2009}).

The tetrablock naturally appears in control engineering and produces problems
of a function-theoretic character. We denote the {\it tetrablock} by $\EE$ and we define it to be the image of the
Cartan domain of the first type $\mathcal R_I:=\mathcal R_I(2,2)=\{x\in\mathcal M(2\times 2,\CC):||x||<1\}$, 
where $||\cdot||$ denotes the operator norm of matrices, under the mapping
$\pi(x):=(x_{11},x_{22},\det x)$. Note also that $\EE$ is the image under $\pi$ of $\mathcal R_{II}:=\mathcal R_{II}(2,2):=\{x\in \mathcal R_I(2,2):x=x^t\}$ (the Cartan domain of the second type). 

In the paper \cite{AWY}
several equivalent definitions of the domain $\EE$ are given.
Recall two of them
\begin{equation}\label{eq:tetrablock}
\EE=\{(z_1,z_2,z_3)\in\CC^3: |z_2-\bar z_1z_3|+|z_1z_2-z_3|+|z_1|^2<1\}\end{equation} and \begin{equation}\label{eq:tetrablock2}
\EE=\{(z_1,z_2,z_3)\in\CC^3: |z_1-\bar z_2z_3|+|z_2-\bar z_1z_3|+|z_3|^2<1\}.
\end{equation}

It is proven in \cite{AWY} that the equality between the Carath\'eodory distance and the Lempert function of $\EE$ with one of the arguments fixed at the origin,
$$
c_{\EE}(0,\cdot)=\tilde k_{\EE}(0,\cdot)
$$
holds on $\EE$, which suggests that the equality between both functions holds on $\EE\times\EE$. 
The question whether these function are equal was also posed explicitly in \cite{AWY}. 
Since both functions are biholomorphically invariant, we also get immediately the equality 
$$
c_{\EE}(z,\cdot)=\tilde k_{\EE}(z,\cdot) \text{ on }\EE
$$
for any $z\in\{\Psi(0): \Psi\in\Aut (\EE)\}=\{(a,b,ab): a,b\in\DD\}$, where $\Aut(\EE)$ is the set of all
biholomorphisms of $\EE$ (for a description of holomorphic automorphisms of $\EE$ see \cite{You 2007} and \cite{Kos 2009}).

The main purpose of the paper is to show that in fact the above equality holds everywhere in the tetrablock thus solving the problem posed in \cite{AWY}. In other words we prove

\begin{thm}\label{thm:1} The equality $c_{\EE}=\tilde k_{\EE}$ holds.
\end{thm}

Since the tetrablock is a hyperconvex domain (and thus taut), in order to prove the above theorem it is equivalent to show that
for any $\tilde k_{\EE}$-extremal $f$ there is a function $F\in\OO(\EE,\DD)$ such that
$F\circ f=\id_{\DD}$ where $\DD$ denotes the unit disc in $\CC$. And this is what we actually do.

Recall that the fundamental Lempert theorem (see \cite{Lem 1981}, \cite{Lem 1982}) states that the Lempert function and the Carath\'eodory
distance coincide on convex domains (and thus on domains that may 
be exhausted by domains biholomorphic to convex ones). Nevertheless, very little was known about the converse theorem in a reasonable class of domains (e.g. in the class of bounded and pseudoconvex domains).
A few years ago C.~Costara, J.~Agler and N.~J.~Young showed that the Lempert theorem would hold in the symmetrized bidisc
(see \cite{Cos 2004}, \cite{Cos 2004b}, \cite{Agl-You 2004}) which is neither biholomorphic to a convex domain
(see \cite{Cos 2004b}) nor can be exhausted by domains biholomorphic to convex ones (see \cite{Edi 2004}).
The symmetrized bidisc is a domain in $\CC^2$ denoted by $\GG_2$ and given by \begin{equation}\label{eq:sym} \GG_2:=\{(s,p)\in \CC^2:\ |s-\bar sp|+|p^2|<1\}.\end{equation} For the basic properties of $\GG_2$ we refer the reader to \cite{Agl-You 2004}.

Following the ideas in the papers \cite{Cos 2004b} and \cite{Edi 2004} we show that the tetrablock has the same properties.

\begin{thm}\label{thm:convex} $\EE$ cannot be exhausted by domains biholomorphic to convex ones.
\end{thm}

The above theorems show that the tetrablock is the second example of that kind.
Recall also that the symmetrized bidisc is a $\CC$-convex domain (see \cite{Nik-Pfl-Zwo 2008}) 
Therefore it is natural to pose the question whether $\EE$ is biholomorphic
to a $\CC$-convex domain (or even more, whether it can be exhausted by domains biholomorphic to $\CC$-convex domains).
And we also may repeat other open questions concerning the $\CC$-convexity. Does the Lempert theorem hold for any bounded $\CC$-convex domain
(see \cite{Zna 2001})? In fact, since the Lempert Theorem holds for all $C^2$-smooth $\CC$-convex domains (see \cite{Jac 2006}),
the positive answer to this question would follow from the positive solution of another problem from \cite{Zna 2001}:
Can any bounded $\CC$-convex domain be exhausted by $C^2$-smooth $\CC$-convex domains?

The tetrablock is an example of a bounded $(1,0,1)$-balanced and $(0,1,1)$-balanced pseudoconvex domain;
recall that if $m_1,\ldots,m_n$ are non-negative integers (to avoid trivialities we assume that at least one of $m_j$'s is non-zero)
then a domain $D\subset\CC^n$ is called 
\emph{$(m_1,\ldots,m_n)$-balanced} if for any $z\in D$ and $\lambda\in\bar\DD$ the point $(\lambda^{m_1}z_1,\ldots,\lambda^{m_n}z_n)$ 
lies in $D$. In the proof of Theorem~\ref{thm:1} we make use of the following result that has a more general formulation 
and is interesting for its own.

\begin{thm}\label{thm:gen}
Let $D$ be an $(m_1,\ldots,m_n)$ balanced pseudoconvex domain.
Assume that $\psi$ is a complex geodesic in $D$ and $\psi(\lambda)=(\lambda^{m_1}\phi_1(\lambda),\ldots,\lambda^{m_n}\phi_n(\lambda))$, $\lambda\in\DD,$ for some $\varphi_j$ holomorphic on $\mathbb D,$ $j=1,\ldots,n$.

Then $\phi\in\OO(\DD,\partial D)$ or $\phi$ is a complex geodesic in $D$.
\end{thm}

Note that one may formulate Theorem~\ref{thm:gen} replacing the geodesics with $\tilde k$-extremals - in this case the proof is immediate.

\bigskip

\textbf{Acknowledgments} The authors express their thanks to the referee for her/his valuable
suggestions which essentially improved the quality of the paper.

\section{Definitions, preliminary considerations and proof of Theorem~\ref{thm:gen}}
First we recall basic definitions of the considered notions. 
The basic properties in the theory of holomorphically invariant functions may be found in \cite{Jar-Pfl 1993}.

For a domain $D\subset\CC^n$, $w,z\in D$, we define the {\it Lempert function}
\begin{equation}
\tilde k_D(w,z):=\inf\{p(\lambda_1,\lambda_2):\text{ there is } f\in\OO(\DD,D), f(\lambda_1)=w,f(\lambda_2)=z\},
\end{equation}
where $p$ denotes the Poincar\'e distance on $\DD$.

If $w\neq z$, then any $f$ as in the definition of $\tilde k_D$ such that $\tilde k_D(w,z)=p(\lambda_1,\lambda_2)$ 
is called a {\it $\tilde k_D$-extremal for $(w,z)$} (or shortly {\it extremal}).

We also define {\it the Carath\'eodory (pseudo)distance} by
\begin{equation}
c_D(w,z):=\sup\{p(F(w),F(z)):F\in\OO(D,\DD)\}.
\end{equation}

It is obvious that $c_D\leq \tilde k_D$. The Lempert Theorem states that if $D$ is convex then $\tilde k_D=c_D$.

The idea of the proof of Theorem~\ref{thm:1} is the following. We take any $\tilde k_{\EE}$-extremal $f$
(and the existence of extremal for any pair of different points from $\EE$ follows from the tautness of $\EE$)
and we show that there is a left inverse to $f$,
i.e. a mapping $F\in\OO(\EE,\DD)$ such that $F\circ f$ is an automorphism
(without loss of generality we may require the automorphism to be the identity). In such a case the mapping $f$ is called a
{\it complex geodesic}.
There will be two kinds of possible choices of $F$ depending on the location of the image of $f$.
Either the image of $f$ intersects $\mathcal T:=\{z\in\CC^3:z_1z_2=z_3\}$ or it is disjoint from $\mathcal T$.
In the second case we can lift the extremal (using the fact that $\pi|_{\mathcal R_{II}\setminus\pi^{-1}(\mathcal T)}$ 
is a holomorphic covering onto $\EE\setminus\mathcal T$). In the first case we cannot use the lifting
coming from the holomorphic covering. Nevertheless, making use of the explicit form of the covering 
we may lift the extremal to the extremal lying
in $\mathcal R_I$. In both cases all the extremals (which are automatically complex geodesics because of the convexity of 
$\mathcal R_I$ and $\mathcal R_{II}$) are known.
So we have a form of possible extremals. Now the left inverse will be of two possible forms. 
Either the one considered in several papers in the case the extremal passes through $\mathcal T$
(see \cite{AWY} and \cite{EZ}) or a function obtained from that in a way described in a more general situation 
in the proof of Theorem~\ref{thm:gen}. Therefore, we start with the proof of that theorem.

\begin{proof}[Proof of Theorem \ref{thm:gen}]
It is clear that $\phi\in\OO(\DD,\partial D)$ or $\phi\in\OO(\DD,D)$. Assume that the second case holds. 
Let $F\in\OO(D,\DD)$ be such that $F\circ\psi=\id_{\DD}$.

We claim that for any $z=(z_1,\ldots,z_n)\in D$ there is exactly one $\lambda=\lambda(z)\in\DD$ such that
$F(\lambda^{m_1}z_1,\ldots,\lambda^{m_n}z_n)=\lambda$. 
In fact, fix $z\in D$ and consider two functions defined on a neighborhood of $\bar\DD$: 
$\lambda\to F(\lambda^{m_1}z_1,\ldots,\lambda^{m_n}z_n)$ and $\lambda\to\lambda$. 
Since $|F(\lambda^{m_1}z_1,\ldots,\lambda^{m_n}z_n)|<
1=|\lambda|$ for all $\lambda\in\partial\DD$, the Rouch\'e theorem implies that the function 
$\DD\owns\lambda\to\lambda-F(\lambda^{m_1}z_1,\ldots,\lambda^{m_n}z_n)$ has
exactly one root in $\DD$ which finishes the proof of our claim. 
This allows us to define well a function $G:D\mapsto\DD$ with $G(z):=\lambda(z)$, $z\in D$.

Since the graph of $G$ which is equal to
\begin{equation}
\{(z,\lambda)\in D\times\DD:F(\lambda^{m_1}z_1,\ldots,\lambda^{m_n}z_n)=\lambda\}
\end{equation}
is an analytic set (for the notion of analytic sets we refer the reader to \cite{Loj 1991}) we get that $G$ is holomorphic (see e.g. \cite{Loj 1991}, Chapter~V, \S~1).
Moreover, it follows from the definition that $G\circ\phi(\lambda)=\lambda$, $\lambda\in\DD$,
which finishes the proof.
\end{proof}

\section{Proof of Theorem 1 -- the case $f(\DD)\cap\mathcal T\neq\varnothing$.}
Let $\Phi_a(x)=(1-aa^*)^{-\frac{1}{2}}(x-a)(1-a^*x)^{-1}(1-a^*a)^{\frac{1}{2}},$ $a,x\in\mathcal R_I.$ It is known (see e.g. \cite{Bas}) that $\Phi_a\in \Aut(\mathcal R_I),$ $\Phi_a(0)=-a$ and $\Phi_a(a)=0.$ If additionally $a,x$ are symmetric, then $\Phi_a(x)$ is symmetric as well. Therefore, $\Phi_a\in \Aut(\mathcal R_{II}),$ $a\in \mathcal R_{II}.$

It follows from \cite{Kos 2009} that for any $\psi\in \Aut(\mathbb E)$ there is a $\Phi\in\Aut(\mathcal R_{II})$ such that \begin{equation}\label{eq:aut} \psi\circ \pi(x) = \pi\circ \Phi(x),\quad x\in \mathcal R_{II}.\end{equation} It is easy to observe that \begin{equation}\label{eq:nowphi} \Phi=U\Phi_a U^t\end{equation} for some $a=\left( \begin{array}{cc} a_1 & 0 \\  0 & a_2 \\ \end{array} \right),$ $a_1,a_2\in\mathbb D,$ and $U=\left(\begin{array}{cc} e^{i\theta} & 0 \\0 & e^{i\eta} \\ \end{array} \right)$ or $U=\left(\begin{array}{cc} 0 & e^{i\theta} \\ e^{i\eta} &  0\\ \end{array} \right),$ $\theta,\eta\in \mathbb R.$ Direct computations show that the equality (\ref{eq:aut}) remains valid on $\mathcal R_I,$ i.e. \begin{equation}\label{eq:autsec} \psi\circ \pi(x) = \pi\circ \Phi(x),\quad x\in \mathcal R_I.\end{equation}

Note also that it follows from (\ref{eq:aut}) that all automorphisms of $\EE$ extend holomorphically onto a neighborhood of $\bar\EE$.

\bigskip

Put $\tilde c=\left(\begin{array}{cc} 0 & 0 \\ c & 0 \\ \end{array} \right),$ where $c\in \mathbb D.$ Let us denote \begin{equation}\label{eq:phic}\varphi_c(x):=\Phi_{\tilde c}(x)=\left( \begin{array}{cc} \sqrt{1-|c|^2}\frac{x_{11}}{1-\overline c x_{21}} & \frac{x_{12}+\overline c\det x}{1-\overline cx_{21}} \\ \frac{x_{21}-c}{1-\overline c x_{21}} & \sqrt{1-|c|^2} \frac{x_{22}}{1- \overline c x_{21}} \\ \end{array} \right),\quad x=(x_{ij})\in \mathcal R_I.
\end{equation} Note that $$\det\varphi_c(x)=\frac{\det x+ cx_{12}}{1- \overline c x_{21}}.$$

We start with the following observation:
\begin{lemma}\label{lem:shilov}
 Let $f:\mathbb D\to \partial \mathbb E$ be an analytic disc. If $f(\mathbb D)\cap \mathcal T\neq \varnothing,$ then $f(\mathbb D)\subset \mathcal T.$
\end{lemma}

\begin{proof}
Using (\ref{eq:tetrablock}) we get \begin{equation}\label{eq:bound}|f_2-\overline{f_1}f_3|+|f_1f_2-f_3|=1-|f_1|^2.\end{equation} Let $\lambda_0$ be such that $f(\lambda_0)\in \mathcal T.$ Then $|f_2(\lambda_0)-\overline{f_1(\lambda_0)}f_3(\lambda_0)|=1-|f_1(\lambda_0)|^2.$ Using the equality $f_1(\lambda_0)f_2(\lambda_0)=f_3(\lambda_0)$ again we infer that $$|f_2(\lambda_0)|(1-|f_1(\lambda_0)^2|)=1-|f_1(\lambda_0)^2|,$$ whence $|f_1|\equiv 1$ or $|f_2|\equiv 1$ (recall that $|f_i|\leq 1,$ $i=1,2,3$). Assume without loss of generality that $|f_1|\equiv 1$. Making use of (\ref{eq:bound}) we find that $f_1f_2=f_3.$
\end{proof}

\begin{defin}
For a holomorphic mapping $f:\mathbb D\to \mathbb E$ put $$\nu(f)(\lambda)=\ord_{\lambda} (f_1f_2-f_3),\quad \lambda\in \mathbb D.$$
\end{defin}

\begin{rem}\label{rem:nu}
Note that $f(\mathbb D) \subset  \mathbb E\setminus\mathcal T$ if and only if $\nu(f)\equiv 0.$ Moreover, $\nu$ is invariant under automorphisms of the tetrablock, i.e.
\begin{equation}\label{eq:invofnu} \nu(f)\equiv \nu(\varphi\circ f),\quad \varphi\in \Aut(\mathbb E).
\end{equation} Actually, it follows from \eqref{eq:aut} that there is an automorphism $\Phi$ of $\mathcal R_{II}$ such that $\varphi(\pi(x))=\pi(\Phi(x))$ for $x\in \mathcal R_{II}.$ Moreover, $\Phi$ is of the form \eqref{eq:nowphi}. Direct calculations show that $\varphi_1(x) \varphi_2(x) - \varphi_3(x)= (x_1x_2- x_3)e^{2i(\eta+\theta)}(1-|a_1|^2)(1-|a_2|^2) (1-\bar a_1 x_1-\bar a_2 x_2+\bar a_1 \bar a_2 x_3)^{-2},$ $x=(x_1,x_2,x_3)\in \mathbb E$, where $\eta,\theta$ and $a$ are as in \eqref{eq:nowphi}. Since $1-\bar a_1 x_1-\bar a_2 x_2+\bar a_1 \bar a_2 x_3=\det(1-a^* y),$ where $y\in \mathcal R_{II}$ is such that $\pi(y)=x,$ we see that the function $ x\mapsto 1-\bar a_1 x_1-\bar a_2 x_2+\bar a_1 \bar a_2 x_3$ does not vanish on $\mathbb E.$ This immediately gives \eqref{eq:invofnu}.
\end{rem}

\begin{lemma}\label{lem:lift}
 Let $f:\mathbb D \to \mathbb E$ be a holomorphic disc such that $f^{-1}(\mathcal T)\neq \varnothing.$ Then there is a holomorphic disc $F:\mathbb D\to \overline{\mathcal R_I}$ such that $f=\pi\circ F.$

Moreover, one of two following possibilities holds:

(a) $F(\mathbb D)\subset \mathcal R_I,$

(b) there is an automorphism $\varphi$ of the tetrablock and a holomorphic mapping $\psi:\mathbb D\to \mathbb D$ such that $f(\lambda)= \varphi ((0,0,\psi(\lambda))),$ $\lambda\in\mathbb D.$
\end{lemma}

\begin{proof}

\emph{Step 1.} First consider the case when $\#f^{-1}(\mathcal T)=1.$ Since the group $\Aut(\EE)$ acts transitively on $\mathcal T,$ losing no generality we may assume that $f(0)=0.$ Then there are $n,m\in \mathbb N,$ $n+m>0,$ such that $$f= (\lambda^n g_1,\lambda^m g_2, \lambda^{n+m} g_3),$$ for some holomorphic $g=(g_1,g_2,g_3):\mathbb D\to \bar\EE,$ $g(0)\neq 0.$ Note that $g(\mathbb D\setminus\{0\})\cap \mathcal T=\varnothing$ and $\nu(g)(0)<\nu(f)(0).$

If $\nu(g)(0)=0$ (i.e. $g(0)\not\in\mathcal T$), then $g_1g_2-g_3$ does not vanish on $\mathbb D.$ Let $\tilde g$ be an analytic square root of $g_1g_2-g_3.$ Then the mapping $G=\left(\begin{array}{cc} g_1 & \tilde g \\ \tilde g & g_2 \\ \end{array} \right):\mathbb D\to \overline{\mathcal R_{II}}$ satisfies $g=\pi \circ G$.  Put $F(\lambda):=\left(\begin{array}{cc} \lambda^n g_1(\lambda)& \lambda^n \tilde g(\lambda) \\ \lambda^m \tilde g(\lambda) & \lambda^m g_2(\lambda) \\ \end{array}
\right),$ $\lambda\in \mathbb D.$ Clearly $F:\mathbb D\to \overline{\mathcal R_I}$ and $f=\pi\circ F$.

If $\nu(g)(0)\neq 0$, then $g(0)\in \mathcal T$ and, by Lemma~\ref{lem:shilov}, $g(\mathbb D)\subset \mathbb E$. Let $\varphi\in \Aut(\mathbb E)$ be such that $\varphi(g(0))=0.$ There is an analytic disc $h:\mathbb D\to \bar\EE$ such that $h(0)\neq 0$ and $$\varphi\circ g= (\lambda^{n_1} h_1,\lambda^{m_1} h_2, \lambda^{n_1+m_1} h_3),$$ $n_1,m_1\in \mathbb N,$ $n_1+m_1>0.$ In view of Remark~\ref{rem:nu} $$\nu(h)(0)<\nu(\varphi\circ g)(0)=\nu(g)(0)<\nu (f)(0).$$ If $\nu(h)(0)=0$ repeating the previous argument we find that there is a mapping $ H:\mathbb D\to \overline{\mathcal R_{II}}$ such that $h=\pi\circ H.$ Therefore, we may construct a mapping $G_1:\mathbb D\to \overline{\mathcal R_I}$ such that $\varphi\circ g=\pi \circ G_1.$ Making use of (\ref{eq:autsec}) we infer that $g=\pi\circ \widehat G$ for some analytic disc $\widehat G=(\widehat g_{ij})$ in $\overline{\mathcal R_I}.$ In particular, $f=\pi\circ F_1,$ where $F_1:\mathbb D\to \overline{\mathcal R_I}$ is given by the formula $F_1(\lambda)=\left(\begin{array}{cc} \lambda^{n_1} \widehat g_{11}(\lambda)& \lambda^{n_1} \widehat g_{12}(\lambda) \\ \lambda^{m_1} \widehat g_{21}(\lambda) & \lambda^{m_1} \widehat g_{22}(\lambda) \\ \end{array}\right),$ $\lambda\in \mathbb D.$ If $\nu (h)(0)>0$ we repeat the above procedure (until $\nu =0$).

\emph{Step 2.} In the case when $f^{-1}(\mathcal T)$ is finite it is sufficient to apply the procedure from the previous step to every point of $f^{-1}(\mathcal T)$.

\emph{Step 3.} Now consider the case when that $f^{-1}(\mathcal T)$ is infinite. If $f(\mathbb D)\subset \mathcal T,$ the statement is clear.
In the other case applying \textit{Step 2} to the family of analytic discs $f|_{(1-1/n)\mathbb D},$ $n\in\mathbb N,$ we find that there are holomorphic mappings $g_n:(1-1/n)\mathbb D\to \overline{\mathcal R_I}$ such that $$f\equiv \pi\circ g_n\quad \text{on}\ (1-1/n)\mathbb D.$$ Taking a limit of a subsequence we get a holomorphic $g:\mathbb D\to \overline{\mathcal R_I}$ such that $$f\equiv\pi\circ g,$$ which finishes the proof of the first assertion.

To prove the second statement assume without loss of generality that $f(0)=0.$ Note that $g(0)=\left(\begin{array}{cc} 0 & 0 \\ c & 0 \\ \end{array} \right)$ or $g(0)=\left(\begin{array}{cc} 0 & c \\ 0 & 0 \\ \end{array} \right)$  for some $c\in \bar\DD.$ If $c\in \partial \mathbb D$ we deduce that $f=(0,0,\psi)$ for some holomorphic mapping $\psi$.
In the case when $c$ lies in the unit disc it is clear that $g(0)\in \mathcal R_I$, whence $g(\mathbb D)\subset \mathcal R_I.$
\end{proof}

Recall that any complex geodesic in $\mathcal R_I$ passing through the origin is of the form \begin{equation}\mathbb D\ni \lambda\to U \left(\begin{array}{cc} \lambda & 0 \\ 0 & Z(\lambda) \\ \end{array} \right)V\in \mathcal R_I,\end{equation} where $U$, $V$ are unitary matrices and $Z:\mathbb D\to \mathbb D$ is a holomorphic mapping such that either $Z(\lambda)=\lambda$ or $|Z(\lambda)|<|\lambda|$ for $\lambda\in \mathbb D\setminus\{0\}$  (see \cite{Abate}).

If $f$ is an extremal function in the tetrablock and $g:\mathbb D\to \mathcal R_I$ is any holomorphic mapping covering $f$ (i.e. $\pi\circ g=f$), then $g$ is extremal as well. This simple observation together with Lemma~\ref{lem:lift} and the description of complex geodesics of the classical Cartan domain of the first type lead to the following statement which is of key importance for our considerations:
\begin{corollary}\label{cor:lift} If $f:\mathbb D\to \mathbb E$ is an extremal mapping such that $f(0)=0$, then either $f(\mathbb D)\subset \mathcal T$ or $f(\lambda)=(0,0,e^{i\theta}\lambda)$ or there are unitary matrices $U,$ $V$ and there is $c\in \mathbb D$ such that \begin{equation}\label{eq:geo}f(\lambda)=\pi(\varphi_c(U\left(\begin{array}{cc} \lambda & 0 \\ 0 & Z(\lambda) \\ \end{array} \right)V)),\end{equation} where $\varphi_c$ is an automorphism of the Cartan domain of the first type given by the formula (\ref{eq:phic}) and $Z:\mathbb D\to \mathbb D$ is a holomorphic mapping. Moreover $|Z(\lambda)|<|\lambda|$, $\lambda\in\mathbb D\setminus\{0\},$ or $Z(\lambda) =\lambda,$ $\lambda\in \mathbb D.$
\end{corollary}

\begin{lemma}\label{lem:boundofE} Let $v=(v_{ij})\in \partial\mathcal R_{I}.$ If $\pi (v)\in \partial \mathbb E,$ then $|v_{12}|=|v_{21}|.$
\end{lemma}

\begin{proof} Seeking a contradiction suppose that $|v_{12}|\neq |v_{21}|.$ Put $\tilde v=\left( \begin{array}{cc} v_{11} & w \\ w & v_{22} \\ \end{array}
\right),$ where $w$ is an arbitrary square root of $v_{12}v_{21}.$ Note that it would suffice to show that \begin{equation}\label{eq:dag}||\tilde v||<||v||.\end{equation} Actually, since $\pi(\tilde v)=\pi (v)$ and $||v||=1$, the inequality~(\ref{eq:dag}) would imply that $\pi(v)\in \mathbb E.$

Let us denote $\rho:=||v||^2=\rho (vv^*)$ and $\tilde{\rho}:=||\tilde v||^2=\rho (\tilde v \tilde v^*).$ Put $d:=\det v=\det \tilde v,$ $t:=\tr (vv^* ) = |v_{11}|^2+|v_{12}|^2+|v_{21}|^2+|v_{22}|^2$ and $\tilde t:= \tr (\tilde v\tilde v^*)= |v_{11}|^2+2|v_{12}||v_{21}|+|v_{22}|^2.$ It is clear that $\tilde t<t.$

Since $\rho=1/2(t+\sqrt{t^2-4d})$ and $\tilde{\rho}=1/2(\tilde t+\sqrt{\tilde t^2-4d}),$ we find that $\tilde{\rho}< \rho,$ which proves (\ref{eq:dag}).
\end{proof}

\begin{proof}[Proof of Theorem~\ref{thm:1} in the case $f(\mathbb D)\cap \mathcal T\neq \varnothing$] 

Let $f$ be an extremal mapping in the tetrablock such that the image of $f$ intersects the set of triangular points. We lose no generality assuming that $f(0)=0.$ Let $\tau,\sigma \in \mathbb D$, $\tau\neq \sigma$ be such that $f$ is extremal for $(f(\tau),f(\sigma)).$ We aim at showing that $f$ is a complex geodesic.

If $f(\lambda)=(0,0,e^{i\theta}\lambda)$, $\lambda\in\mathbb D,$ the statement is clear. The case $f(\mathbb D)\subset \mathcal T$ follows from \cite{Abo 2007}, Corollary~6.9. Therefore, using Corollary~\ref{cor:lift}, we may assume that $f$ is of the form (\ref{eq:geo}).

First we consider the case when $Z(\lambda)=\lambda,$ $\lambda\in \mathbb D.$ Then $f(\lambda)=\pi(\varphi_c(W\lambda)),$ $\lambda\in \mathbb D,$ where $W=UV$ is unitary. Making use of the formula \eqref{eq:phic} we find that $f=(\alpha,\omega\alpha,\gamma),$ where $\alpha(\lambda)=\sqrt{1-|c|^2}w_{11}\lambda/(1-\bar c \lambda w_{21})$, $\gamma(\lambda)=(\det W\lambda^2+c\lambda w_{12})/(1-\bar c \lambda w_{21})$, $\lambda\in \DD,$ and $\omega\in \partial \DD$ is such that $w_{22}=\omega w_{11}.$ Since the tetrablock is $(0,1,1)$-balanced we may assume that $\omega=1.$

The descriptions \eqref{eq:tetrablock2} of the tetrablock and \eqref{eq:sym} of the symmetrized bidisc give us the embedding $$\GG_2\owns(s,p)\mapsto (s/2,s/2,p)\in\EE.$$ Since $f$ is extremal, one can see that $\tilde f:=(2\alpha,\gamma)$ is extremal in $\GG_2.$ Therefore, it follows from \cite{Agl-You 2004} that $\tilde f$ is a geodesic in $\GG_2$ and its left inverse is given by $$F_a(s,p)=\frac{2ap-s}{2-as},\quad (s,p)\in\GG_2,$$ for some $a\in \partial \DD.$ Put $$\Psi_z(x):=\frac{zx_3-x_1}{1-zx_2}\quad x\in \EE,$$ where $z\in \bar\DD,$ and recall that $|\Psi_z|<1$ on $\EE$ whenever $z\in \bar\DD$ (see \cite{AWY}, Theorem~2.1).  It follows from the above considerations that $\Psi_a(f(\lambda))=F_a(\tilde f(\lambda))=\lambda$, $\lambda\in \DD$, so $\Psi_a$ is a left inverse of $f$, whence $f$ is a complex geodesic.

Now we focus on the case when $|Z(\lambda)|<|\lambda|$ for $\lambda\in \mathbb D\setminus\{0\}$ and $c\neq 0.$ It is seen that there is an open neighborhood $D$ of $\bar\DD$ and a holomorphic, non-rational mapping $W:D\to \mathbb C$ such that $W(\mathbb D)\subset \DD,$ $W(\tau)=Z(\tau)$ and $W(\sigma)=Z(\sigma)$ (note that we do not demand $W(0)=0$).

Put $g(\lambda)=\pi(\varphi_c(U\left(\begin{array}{cc} \lambda & 0 \\ 0 & W(\lambda) \\ \end{array} \right)V)),$ $\lambda\in \mathbb D.$ Then $g$ is also an extremal function in the tetrablock (as $g(\sigma)=f(\sigma),$ $g(\tau)=f(\tau)$). Therefore $g$ is almost proper, that is $g^*(\lambda)\in \partial \mathbb E$ ($g^*$ denotes a nontangential limit of the mapping $g$) for almost all $\lambda\in \partial\mathbb D$ w.r.t the Lebesgue measure on the unit circle (see e.g. \cite{Edi-Klis 2010}). Since $g$ is holomorphic in a neighborhood of $\bar \DD$, the almost properness means that $g(\partial \mathbb D)\subset \partial \mathbb E$. 

It follows from Lemma~\ref{lem:boundofE} that $$|\lambda u_{21} v_{11}+ W(\lambda) u_{22}v_{21}-c| = |\lambda u_{11} v_{12}+ W(\lambda) u_{12}v_{22}+\overline c\lambda W(\lambda) \det U \det V|,$$ $\lambda\in \partial\mathbb D.$ We claim that there are finite Blaschke products $B_1$, $B_2$ such that $|B_1(0)|+|B_2(0)|\neq 0$ and 
\begin{equation}\label{eq:B1B2} B_1(\lambda)(\lambda u_{21} v_{11}+ W(\lambda) u_{22}v_{21}-c) = B_2(\lambda)(\lambda u_{11} v_{12}+ W(\lambda) u_{12}v_{22}+\overline c\lambda W(\lambda)e^{i\theta}),\end{equation} $\lambda\in \mathbb D,$ where $e^{i\theta}=\det U\det V.$ 

To prove the existence of such Blaschke products take any $f_1,f_2\in \mathcal O(D)$ not vanishing identically and satisfying $|f_1|=|f_2|$ on $\partial \mathbb D$. Put $H(\lambda):=(\lambda-\lambda_1)\cdots (\lambda-\lambda_N)$, $\lambda\in D$, where $\lambda_1,\ldots, \lambda_N$ are common roots of $f_1$ and $f_2$ lying in $\bar\DD$ and counted with multiplicities. Since $|f_1|=|f_2|$ on $\partial \mathbb D$ we see that $f_1/H$ and $f_2/H$ do not vanish on $\partial \mathbb D.$ Therefore, there are finite Blaschke products $\tilde B_i,$ $i=1,2,$ with no common zeros such that that $F_i:=f_i/(H\tilde B_i)$ is holomorphic on a neighborhood of $\bar\DD$ and does not vanish there, $i=1,2$. Since $|\tilde B_i|=1$ on $\partial \mathbb D$ we get that $|F_1|=|F_2|$ on $\partial \mathbb D.$ From this we immediately get that $F_2/F_1$ is constant -- apply the maximum principle to $F_1/F_2$ and $F_2/F_1$. Let $F_2=\omega F_1,$ where $|\omega|=1$. Then $f_1=F_1 H \tilde B_1,$ $F_2=\omega F_1 H \tilde B_2,$ and $\tilde B_1,$ $\tilde B_2$ have no common zeros. Putting $B_1:=\omega\tilde B_2$ and $B_2:= \tilde B_1$ we obtain Blaschke products satisfying the desired claim.

Since $W$ is non-rational we infer that
\begin{align}\label{eq:B} B_1(\lambda)\lambda u_{21}v_{11}-cB_1(\lambda)&=B_2(\lambda)\lambda u_{11}v_{12},\\\nonumber
B_1(\lambda) u_{22} v_{21} &= B_2(\lambda) u_{12}v_{22}+ B_2(\lambda)\overline c\lambda e^{i\theta}, \; \lambda\in\bar\DD.
\end{align} Putting $\lambda=0$ we get $B_1(0)=0$. Since $B_2(0)\neq 0$ we get that $u_{12}v_{22}=0.$

If $u_{12}=0,$ then $u_{21}=0$ and $|u_{11}|=|u_{22}|=1$. Putting it to \eqref{eq:B} and taking $|\lambda|=1$ we find that $|v_{12}|=|c|.$ Easy computations give: $|v_{11}|=|v_{22}|=\sqrt{1-|c|^2}$ and $|v_{21}|=|c|.$ We want to show that $f$ is a complex geodesic. Note that making use of the fact that the tetrablock is $(1,0,1)$- and $(0,1,1)$-balanced it suffices to get the statement under the additional assumption that $u_{11}=u_{22}=1$.  Using similar argument we see that it is enough to consider that case $v_{11}=\sqrt{1-|c|^2}$ and $v_{12}=|c|.$ Then $$V=\left(\begin{array}{cc}\sqrt{1-|c|^2}& |c|\\ -\omega |c|&\omega \sqrt{1-|c|^2}\end{array}\right)$$ for some $\omega$ from the unit circle. Replacing $Z$ with $\omega Z$ we may clearly assume that $\omega=1.$ Now, after some simple calculations one can deduce that $$f(\lambda)=\left(\frac{(1-|c|^2)\lambda}{1+\bar c|c|Z(\lambda)},\frac{(1-|c|^2)Z(\lambda)}{1+\bar c|c|Z(\lambda)}, \lambda\frac{Z(\lambda)+c|c|}{1+\bar c |c|Z(\lambda)}\right),\quad \lambda\in \DD.$$ Therefore $f$ is a complex geodesic (it just of the form (2) in Theorem~2 in \cite{EZ}, with permuted two first variables, $\omega_2=1$, $\omega_1\in \partial\DD$ such that $c=-\omega_1 |c|,$ $C=|c|^2$ and $\varphi(\lambda)=(\bar\omega_1 Z(\lambda)-|c|^2)/(1-|c|^2\bar\omega_1 Z(\lambda))$). If $v_{22}=0$ we proceed similarly.

Let us focus on the case $c=0$ and $|Z(\lambda)|<|\lambda|,$ $\lambda\in \mathbb D\setminus\{0\}.$ First note that we may assume that $Z$ extends holomorphically to a neighborhood of $\DD.$ Let $h$ be a holomorphic function in a heighborhood of $\DD$ such that $Z(\lambda)=\lambda h(\lambda)$, $h(\DD)\subset\DD$. Replacing $Z$ with a non-rational $W:\mathbb D\to \mathbb D$ (holomorphic on a neighborhood of $\bar\DD$) such that $W(\sigma)=Z(\sigma)$ and $W(\tau)=Z(\tau)$, making use of Lemma~\ref{lem:boundofE} and repeating the argument with Blaschke products \eqref{eq:B} we find that $|u_{21}v_{11}|=|u_{11}v_{12}|$ and $|u_{22}v_{21}|=|u_{12}v_{22}|.$ Since $U$ and $V$ are unitary we deduce from these equalities that 
\begin{equation}\label{eq:reluv} |u_{ij}|=|v_{ij}|,\quad 1\leq i,j\leq 2.\end{equation} Put $$\Phi(\lambda):=U\left(\begin{array}{cc}\lambda & 0\\ 0& Z(\lambda)\end{array}\right)V=\left(\begin{array}{cc}\lambda u_{11}v_{11}+Z(\lambda)u_{12}v_{21}& \lambda u_{11}v_{12}+Z(\lambda)u_{12}v_{22}\\ \lambda u_{21}v_{11}+Z(\lambda)u_{22}v_{21} &\lambda u_{21}v_{12}+Z(\lambda)u_{22}v_{22}\end{array}\right)$$ for $\lambda$ lying in some neighborhood of $\bar\DD$. Obviously $\Phi(\mathbb D)\subset\mathcal R_I$. Define \begin{equation}\label{eq:psi1}\Psi(\lambda)=\left(\begin{array}{cc}\lambda u_{11}v_{11}+Z(\lambda)u_{12}v_{21}&  \lambda^2 u_{11}v_{12}+\lambda Z(\lambda)u_{12}v_{22}\\ u_{21}v_{11}+h(\lambda)u_{22}v_{21} &\lambda u_{21}v_{12}+Z(\lambda)u_{22}v_{22}\end{array}\right)\end{equation} for $\lambda$ from some neighborhood of $\bar\DD$. Note that $||\Psi(\lambda)||=1$ for $\lambda\in \partial \DD$ (as $\Psi\Psi^*$ and $\Phi\Phi^*$ have the same eigenvalues on $\partial \mathbb D$) and $\Psi(0)\in \mathcal R_I$. A standard argument implies that $\Psi$ maps the unit disc into $\mathcal R_I$ (apply the maximum principle to the subharmonic function $\log||\Psi(\cdot)||$). Observe that $f=\pi\circ\Psi$, whence $\Psi$ is a complex geodesic in $\mathcal R_I,$ as well. Denote $c'=-u_{21}v_{11}-h(0)u_{22}v_{21}.$ If $c'=0$ then $u_{21}v_{11}=0$ (remember that $|h(0)|<1$ and use the equality $|u_{21}v_{11}|=|u_{22}v_{21}|$), whence $U$ and $V$ are diagonal or anti-diagonal. Then, it is easy to observe that $f$ is a complex geodesic (more precisely,  up to a permutation of two fist components the mapping $f$ is of the form $f(\lambda)=(\omega_1\lambda,\omega_2 Z(\lambda), \omega_1\omega_2 Z(\lambda)),$ $\lambda\in \mathbb D$, for some $\omega_1,\omega_2\in \partial \mathbb D$).

If $c'\neq 0$, then $u_{21}v_{11}\neq0$. Moving $\Psi(0)$ to the origin and making use of the description of complex geodesics in $\mathcal R_I$ we infer that there are unitary matrices $U_1,V_1$ and a mapping $Z_1$ defined on $\DD$,
such that \begin{equation}\label{eq:psi2}\Psi(\lambda)=\varphi_{c'}(U_1\left(\begin{array}{cc}\lambda& 0\\ 0&Z_1(\lambda)\end{array}\right)V_1),\quad \lambda\in \mathbb D\end{equation} and $|Z_1(\lambda)|<|\lambda|$ for $\lambda\in \mathbb D\setminus\{0\}$ or $Z_1(\lambda)=\lambda,$ $\lambda\in \DD.$ Now we are in a position that allows us to apply the cases already solved (either $Z(\lambda)=\lambda$ for $\lambda\in \DD$ or $c\neq 0$ and $|Z(\lambda)|<|\lambda|$ for $\lambda\in \DD\setminus \{0\}$).
\end{proof}

\section{Proof of Theorem 1 -- the case $f(\DD)\cap\mathcal T=\varnothing$.}

Let $f:\DD\to\EE$ be an extremal such that $f(\DD)\cap
\mathcal T=\varnothing$. Then there exists a geodesic $\widetilde
f:\DD\to \mathcal R_{II}$ such that $f=\pi\circ\widetilde f$. Making use of the form of automorphisms of $\EE$
without loss of generality we may assume that
$f$ is a $\tilde k_{\EE}$-extremal for $(f(0),f(\sigma))$ and $f(0)=(0,0,-\beta^2)$.

Any complex geodesic in $\mathcal R_{II}$ passing through the origin can be
written as
\begin{equation}
\phi(\lambda)=
U
\left(
  \begin{array}{cc}
    \lambda & 0 \\
    0 & Z(\lambda) \\
  \end{array}
\right)
 U^t,
\end{equation}
where $U$ is a unitary matrix and $Z:\DD\to\DD$ is a
holomorphic mapping such that $Z(0)=0$. Moreover,
$|Z(\lambda)|<|\lambda|$, $\lambda\in\DD\setminus\{0\}$, or
$Z(\lambda)=\lambda$ (see \cite{Abate}). Assume that
\begin{equation}
U=
\left(
  \begin{array}{cc}
    a & b \\
    c & d \\
  \end{array}
\right),
\end{equation}
where $|a|^2+|b|^2=|c|^2+|d|^2=1$ and $a\overline c+b\overline
d=0$. After some simple calculations we get
\begin{equation}
\phi(\lambda)=
\left(
  \begin{array}{cc}
    a^2\lambda+b^2 Z(\lambda) & ac\lambda+bd Z(\lambda) \\
    ac\lambda+bd Z(\lambda) & c^2\lambda+d^2 Z(\lambda) \\
  \end{array}
\right).
\end{equation}
Put $A(\lambda)=a^2\lambda+b^2 Z(\lambda)$,
$B(\lambda)=ac\lambda+bd Z(\lambda)$, and $C(\lambda)=c^2+d^2
Z(\lambda)$. We "move" now this geodesic to $\left(
  \begin{array}{cc}
    0 & \beta \\
    \beta & 0 \\
  \end{array}
\right)$ and get the following
\begin{prop}\label{prop:nonsing} Let $f:\DD\to\EE$ be an extremal mapping for
$(f(0),f(\sigma))$ such that $f(0)=(0,0,-\beta^2)$ and
$f(\DD)\cap \mathcal T=\varnothing$. Then there exist
$a,b,c,d\in\bar\DD$ with $|a|^2+|b|^2=|c|^2+|d|^2=1$ and
$a\overline c+b\overline d=0$ such that
\begin{equation}
f(\lambda)=\left(
\frac{A(\lambda)(1-\beta^2)}{\Delta(\lambda)},
\frac{C(\lambda)(1-\beta^2)}{\Delta(\lambda)},
\frac{A(\lambda)C(\lambda)-(B(\lambda)+\beta)^2}{\Delta(\lambda)}
\right),
\end{equation}
where $A,B,C$ are defined as above and
$\Delta(\lambda)=(1+\beta
B(\lambda))^2-A(\lambda)C(\lambda)\beta^2$.
\end{prop}

We show that under the above assumptions the extremal $f$ has its left inverse. First note that the following equations are satisfied: $|a|=|d|$, $|b|=|c|$ and $|ad-bc|=1$. Note also that we may always assume that $Z(\lambda)=\mu\lambda$ for some $|\mu|\leq 1$. Actually, if the considered
mapping is extremal with some $Z$ as above then it will also be extremal with $Z(\lambda)=\mu\lambda$ where $\mu=Z^{\prime}(0)$. If the new considered extremal intersects $\mathcal T$ then in view of the previous considerations we already know that it is a complex geodesic. Therefore, we lose no generality assuming that the extremal omitting $\mathcal T$ is the one with $Z(\lambda)=\mu\lambda$.

We want to get some relations on the numbers $a,b,c,d$ and $Z$ (equivalently, $\mu$) that allow us to describe the mappings as in Proposition
\ref{prop:nonsing}.

When does the equality $f_1f_2=f_3$ hold at some point of $\DD$
(in other words we want to see when $f(\DD)\cap \mathcal T=\varnothing$)?

\begin{equation}
 \frac{AC(1-\beta^2)^2}{\triangle^2}=\frac{AC-(B+\beta)^2}{\triangle}
\end{equation}

which is equivalent to

\begin{equation}
 AC\beta=(B+\beta)(1+\beta B).
\end{equation}

Consequently,

\begin{equation}
 \beta\lambda Z(ad-bc)^2-(1+\beta^2)(\lambda ac+Zbd)-\beta=0.
\end{equation}

Recall that the Cohn criterion (see e.g. \cite{Rah 2002} ) gives that the equation $a_0\lambda^2+a_1\lambda+a_2=0$ has both solutions in $\CC\setminus\DD$
iff $|a_2|\ge|a_0|$ and $|\bar a_0a_1-a_2\bar a_1|\leq||a_0|^2-|a_2|^2|$.

When we apply it to our situation ($Z(\lambda)=\mu\lambda$) we get that $f$ is as desired iff

\begin{equation}
 (1+\beta^2)|\bar\mu(\bar a\bar d-\bar b\bar c)^2(ac+\mu bd)+(\bar a\bar c+\bar\mu\bar b\bar d)|\leq\beta(1-|\mu|^2).
\end{equation}

Then elementary calculations give that the last inequality (remember about the existing relations) is equivalent to
$|c||d|(1+\beta^2)\leq\beta$.

It is sufficient to show that we have the left inverse under the sharp inequality.

In view of Theorem \ref{thm:gen} it is sufficient to show that for some $\gamma\in\DD$ and $|\tau|=1$ the function
\begin{equation}
g:\DD\owns\lambda\mapsto(\tau\frac{\lambda-\gamma }{1- \bar\gamma\lambda }f_1(\lambda),f_2(\lambda),\tau\frac{\lambda-\gamma}{1-\bar\gamma\lambda}
f_3(\lambda))\in\EE
\end{equation}
is a geodesic.

Let $F(z):=\frac{z_3-z_2}{z_1-1}$, $z\in\EE$. We shall prove that by the proper choice of $\tau$ and $\gamma$
the function $h:=F\circ g$ is an automorphism of $\DD$. But it is sufficient, by the Schwarz-Pick Lemma
to show that $|h^{\prime}(0)|=1-|h(0)|^2$.

But $h(0)=-\tau\gamma\beta^2$ and

\begin{multline}
h^{\prime}(0)=\tau\beta^2(1-|\gamma|^2)-2\tau\beta\gamma(ac+bd\mu)(1-\beta^2)+\\ (c^2+d^2\mu)(1-\beta^2)+(c^2+d^2\mu)(1-\beta^2)+
\tau^2\gamma^2(a^2+b^2\mu)(1-\beta^2)\beta^2.
\end{multline}

Consequently,

\begin{equation}
h^{\prime}(0)=(1-\beta^2)((c-\tau\beta\gamma a)^2+\mu(d-b\tau\beta \gamma)^2)+\tau\beta^2(1-|\gamma|^2).
\end{equation}
We choose $|\tau|=1$ and $\gamma\in\DD$ such that $d=b\tau\beta\gamma$ and $|h^{\prime}(0)|=|(1-\beta^2)|c-\tau\beta\gamma a|^2+\beta^2(1-|\gamma|^2)$,
which is possible under the assumption $|d|^2<|b|^2\beta^2$, which is equivalent ($|b|=|c|$ and $|c|^2+|d|^2=1$) to
$\frac{1}{1+\beta^2}<|c|^2$. And the last inequality is equivalent to $|c|^2(1-|c|^2)<\frac{\beta^2}{(1+\beta^2)^2}$.

\section{$\EE$ cannot be exhausted by domains biholomorphic to convex ones}

In this Section we prove Theorem~\ref{thm:convex}.

For $z\in \mathbb C^3$ put $\rho(z):=\max ||(\pi|_{\mathcal R_{II}})^{-1}(z)||.$ The properness of $\pi|_{\mathcal R_{II}}$ implies that $\rho$ is plurisubharmonic.

\begin{proof}[Proof of Theorem~\ref{thm:convex}]
For any $\epsilon\in(0,1)$ we define
$G_{\epsilon}:=\{z\in\CC^3:\rho(z)<1-\epsilon\}$. Assume that
$U_{\epsilon}$ is a neighborhood of $\overline{G_{\epsilon}}$
and $f_{\epsilon}:U_{\epsilon}\mapsto V_{\epsilon}$ where
$V_{\epsilon}$ is a convex domain. Without loss of generality
we may assume that $0\in V_{\epsilon}$, $V_{\epsilon}$ is a
convex domain, $f_{\epsilon}(0)=0$,
$f_{\epsilon}^{\prime}(0)=\operatorname{id}$. Denote
$\phi_{\lambda}(z):=(\lambda z_1,\lambda z_2,\lambda^2 z_3)$,
$\lambda\in\CC$, $z\in \EE$.

Fix $w=(w_1,w_2,w_3),z=(z_1,z_2,z_3)\in\CC^3$ and $r\in[0,1]$.
Put

\begin{enumerate}
\item $R:=\max\{\rho(w),\rho(z)\}$,
\item
    $g_{\epsilon}(\lambda):=f_{\epsilon}^{-1}(rf_{\epsilon}(\phi_{\lambda}(w))+(1-r)f_{\epsilon}(\phi_{\lambda}(z)))$.
\end{enumerate}

Note that $g_{\epsilon}(0)=0$ and that $g_{\epsilon}$ is
well-defined for $|\lambda|<(1-\epsilon)/R$. Moreover,
$\rho(g_{\epsilon}(\lambda))\leq 1$ for any
$|\lambda|<(1-\epsilon)/R$. Put
$h_{\epsilon}(\lambda):=\phi_{1/\lambda}(g_{\epsilon}(\lambda))$.
Then
$h_{\epsilon}:\DD(0,(1-\epsilon)/R)\setminus\{0\}\mapsto\CC^3$
is a holomorphic mapping. Then simple calculations show the
following properties

\begin{enumerate}
\item $(g_{\epsilon})_j^{\prime}(0)=rw_j+(1-r)z_j,\; j=1,2,$
\item $(g_{\epsilon})_3^{\prime}(0)=0$.
\end{enumerate}

Consequently, $h_{\epsilon}$ extends holomorphically to $0$.
More calculations show that

\begin{multline}
(g_{\epsilon})_3^{\prime\prime}(0)=
2(rw_3+(1-r)z_3)+\frac{\partial^2(f_{\epsilon})_3}{\partial
z_1^2}(0)r(1-r)(w_1-z_1)^2 +\\
\frac{\partial^2(f_{\epsilon})_3}{\partial
z_2^2}(0)r(1-r)(w_2-z_2)^2+
2\frac{\partial^2(f_{\epsilon})_3}{\partial z_1\partial
z_2}(0)r(1-r)(w_1-z_1)(w_2-z_2).
\end{multline}

Define

\begin{equation}
s_{\epsilon}:=\frac{1}{2}\frac{\partial^2(f_{\epsilon})_3}{\partial z_1^2}(0),\; t_{\epsilon}:=
\frac{1}{2}\frac{\partial^2(f_{\epsilon})_3}{\partial z_2^2}(0),\;u_{\epsilon}:=\frac{\partial^2(f_{\epsilon})_3}{\partial z_1\partial z_2}(0).
\end{equation}

Then

\begin{multline}\label{mul:h0}
h_{\epsilon}(0)=(rw_1+(1-r)z_1,rw_2+(1-r)z_2,
rw_3+(1-r)z_3+s_{\epsilon}r(1-r)(w_1-z_1)^2+\\
t_{\epsilon}r(1-r)(w_2-z_2)^2+u_{\epsilon}r(1-r)(w_1-z_1)(w_2-z_2)).
\end{multline}

By the maximum principle \begin{equation}\label{eq:h0}
\rho(h_{\epsilon}(\lambda))=\rho(\phi_{1/\lambda}(g_{\epsilon}(\lambda))=\frac{1}{|\lambda|}\rho(g_{\epsilon}(\lambda))\leq\frac{1}{|\lambda|}.
\end{equation}

Hence, $\rho(h_{\epsilon}(0))\leq\frac{R}{1-\epsilon}$.

Our next aim is to show that

\begin{equation}\label{eq:stu}
\lim\sb{\epsilon\to 0}s_{\epsilon}=\lim\sb{\epsilon\to 0}t_{\epsilon}=\lim\sb{\epsilon\to 0}u_{\epsilon}=0.
\end{equation}

Note that the equalities (\ref{eq:stu}) imply that
$\rho(rw+(1-r)z)\leq\max\{\rho(w),\rho(z)\}$ for all
$w,z\in\CC^3$, which contradicts the non-convexity of
$\EE$.

We are just left with the proof of the above equalities.

Put $r=\frac{1}{2}$. For the proof of the convergence of $s_{\epsilon}$ consider two points
$w=(1,1,1)$,
$z=(-1,1,-1).$ Putting them to (\ref{mul:h0}) and (\ref{eq:h0}) we find that $\rho(0,1,s_{\epsilon})\leq 1/(1-\epsilon).$ This inequality implies that $s_\epsilon \to 0.$

Similarly, putting $w=(1,1,1)$ and $z=(1,-1,-1)$ one can show that $t_{\epsilon}\to 0.$

Finally, taking $z=(\zeta,\zeta,\zeta)$ and $w=(-\zeta,-\zeta,\zeta)$, where $|\zeta|=1$ is such that $u_{\epsilon}\zeta =|u_{\epsilon}|$ we find that $\rho(0,0 ,\zeta (1 + \zeta t_{\epsilon} + \zeta s_{\epsilon}+ |u_{\epsilon}|))< 1/(1-\epsilon).$ Making use of just proven two equalities we get the equality $\lim_{\epsilon\to 0} u_{\epsilon}=0.$
\end{proof}

\end{document}